\documentclass[a4paper,11pt]{article}

\usepackage{a4wide}
\usepackage[utf8]{inputenc}
\usepackage[T1]{fontenc}
\usepackage{amsmath,amsthm,amsfonts,bbm}
\usepackage{graphicx}

\newtheorem{theorem}{Theorem}

\newtheorem{lemma}{Lemma}

\newtheorem{remark}{Remark}

\newcommand{\R}{\mathbb{R}}
\newcommand{\N}{\mathbb{N}}

\newcommand{\Kreg}{{K_{\mathrm{reg}}}}
\newcommand{\Klog}{{K_{\log}}}

\newcommand{\abs}[1]{\left\lvert{#1}\right\rvert}

\newcommand{\norm}[1]{{\left\lVert{#1}\right\rVert}}

\renewcommand{\vec}[1]{\underline{#1}}

\begin{document}

\title{Integral Representations and Quadrature Schemes for the Modified Hilbert Transformation}
\author{Marco~Zank}
\date{
        Fakult\"at f\"ur Mathematik, Universit\"at Wien, \\
        Oskar-Morgenstern-Platz 1, 1090 Wien, Austria \\[1mm]
        {\tt marco.zank@univie.ac.at}
      }
      
\maketitle

\begin{abstract}
    We present quadrature schemes to calculate matrices, where the so-called modified Hilbert transformation is involved. These matrices occur as temporal parts of Galerkin finite element discretizations of parabolic or hyperbolic problems when the modified Hilbert transformation is used for the variational setting. This work provides the calculation of these matrices to  machine precision for arbitrary polynomial degrees and non-uniform meshes. The proposed quadrature schemes are based on weakly singular integral representations of the modified Hilbert transformation. First, these weakly singular integral representations of the modified Hilbert transformation are proven. Second, using these integral representations, we derive quadrature schemes, which treat the occurring singularities appropriately. Thus, exponential convergence with respect to the number of quadrature nodes for the proposed quadrature schemes is achieved. Numerical results, where this exponential convergence is observed, conclude this work.
\end{abstract}

\section{Introduction}

For the discretization of parabolic or hyperbolic evolution equations, the classical approaches are time stepping schemes together with finite element methods in space, see, e.g., \cite{ThomeeBuch2006} for parabolic and \cite{BangerthGeigerRannacher2010} for hyperbolic problems. An alternative is to discretize the time-dependent problem without separating the temporal and spatial variables, i.e., space-time methods, see, e.g., \cite{SteinbachLangerBuch2019}. In general, the main advantages of space-time methods are space-time adaptivity, space-time parallelization and the treatment of moving boundaries. However, space-time approximation methods depend strongly on the space-time variational formulations on the continuous level. Different variational approaches for parabolic and hyperbolic problems are contained in \cite{Ladyzhenskaya1985, LM11968}.  In recent years, novel variational settings, which use a so-called modified Hilbert transformation $\mathcal H_T$, have been introduced, see \cite{SteinbachZankETNA2020, ZankDissBuch2020} for parabolic problems and see \cite{LoescherSteinbachZankDD, LoescherSteinbachZankWelleTheorie} for hyperbolic problems. Conforming space-time discretizations by piecewise polynomials of these variational formulations lead to huge linear systems. In \cite{LangerZankSISC2021, ZankWaermeLoeserICOSAHOM2022, ZankWelleLoeser}, fast space-time solvers are developed, where some of them allow for an easy parallelization in time. Further extensions for this variational setting, which includes the modified Hilbert transformation $\mathcal H_T$, are an $hp$-FEM in temporal direction and a classical FEM with graded meshes in spatial direction, see \cite{PerugiaSchwabZankhp2022}. All the mentioned discretization schemes ask for an accurate realization of the modified Hilbert transformation $\mathcal H_T$, which acts only on the temporal part, to calculate the corresponding matrices. There are different possibilities of computing the entries of these temporal matrices. In \cite{SteinbachZankJNUM2021}, a weakly singular integral representation of the modified Hilbert transformation $\mathcal H_T$ for sufficiently smooth functions and quadrature schemes are derived, which can be used for discretizations of parabolic problems as in \cite{SteinbachZankETNA2020, ZankDissBuch2020}. However, in \cite{SteinbachZankJNUM2021}, singular integrals are proposed to be calculated analytically, which is not convenient for high-order approaches, as the afore-mentioned $hp$-FEM. In addition, for discretizations of hyperbolic problems as in \cite{LoescherSteinbachZankDD, LoescherSteinbachZankWelleTheorie}, the weakly singular integral representation of the modified Hilbert transformation $\mathcal H_T$, given in \cite{SteinbachZankJNUM2021}, is not applicable. One way out is the approach in \cite{ZankCMAM2021}, which is well-suited for low polynomial degrees. Thus, accurate realizations of the modified Hilbert transformation $\mathcal H_T$ for high polynomial degrees, as needed in an $hp$-FEM, seem not to be available. In this work, we provide such realizations. For this purpose, we prove a new weakly singular integral representation of the modified Hilbert transformation $\mathcal H_T$ for sufficiently smooth functions with jumps and derive quadrature schemes, which allow for the calculation of the involved matrices to high float point accuracy.

In greater detail, for given $T>0$, $\mu \geq 0$ and $f$, we consider the ordinary differential equation (ODE)
\begin{equation} \label{Einf:Waerme}
  \partial_t u(t) + \mu u(t) = f(t) \, \text{ for } t \in (0,T), \quad u(0) = 0,
\end{equation}
which is the temporal part of the heat equation, and the ordinary differential equation
\begin{equation} \label{Einf:Welle}
  \partial_{tt} u(t) + \mu u(t) = f(t) \, \text{ for } t \in (0,T), \quad u(0) = \partial_t u(0) = 0,
\end{equation}
which is the temporal part of the wave equation. Note that $\mu$ plays the role of an eigenvalue of the spatial operator, see \cite{Ladyzhenskaya1985, SteinbachZankETNA2020, ZankDissBuch2020}.

To state related variational formulations to \eqref{Einf:Waerme}, \eqref{Einf:Welle}, we introduce classical Sobolev spaces, where details and further references are contained in \cite{ZankDissBuch2020}. The classical Sobolev space $H^1(0,T)$ is endowed with the norm $\norm{\cdot}_{H^1(0,T)}$. Its closed subspaces
\begin{equation*}
    H^1_{0,}(0,T) = \{ v \in H^1(0,T) : \, v(0)=0 \} \quad \text{ and } \quad H^1_{,0}(0,T) = \{ v \in H^1(0,T) : \, v(T)=0 \}
\end{equation*}
are endowed with the Hilbertian norm $\abs{\cdot}_{H^1(0,T)} := \norm{\partial_t \cdot}_{L^2(0,T)},$ where $\norm{\cdot}_{L^2(0,T)}$ is the usual $L^2(0,T)$ norm and $\partial_t$ denotes the (weak) derivative. Additionally, we have the Sobolev spaces
\begin{align*}
  H^{1/2}_{0,}(0,T) &:=  \left\{ U_{|(0,T)} \colon \, U \in H^{1/2}(-\infty,T) 
  \text{ with } U(t) = 0, \, t < 0 \right\}, \\
  H^{1/2}_{,0}(0,T) &:=  \left\{ U_{|(0,T)} \colon \, U \in H^{1/2}(0,\infty) 
  \text{ with } U(t) = 0, \, t > T \right\}
\end{align*}
with the Hilbertian norms
\begin{align*}
  \| v \|_{H^{1/2}_{0,}(0,T)} := 
  \left( \norm{v}_{H^{1/2}(0,T)}^2 + \int_0^T \frac{\abs{v(t)}^2}{t} \mathrm dt \right)^{1/2}, \\
  \| v \|_{H^{1/2}_{,0}(0,T)} := \left( \norm{v}_{H^{1/2}(0,T)}^2 + \int_0^T 
  \frac{\abs{v(t)}^2}{T-t} \mathrm dt \right)^{1/2},
\end{align*}
where $\norm{\cdot}_{H^{1/2}(J)}$ is a norm, e.g., the Slobodetskii norm, in the usual Sobolev space $H^{1/2}(J)$ for open intervals $J \subset \R$. Further, $\langle \cdot , \cdot \rangle_{(0,T)}$ denotes the duality pairing in $[H^{1/2}_{,0}(0,T)]'$ and $H^{1/2}_{,0}(0,T)$, or in $[H^1_{,0}(0,T)]'$ and $H^1_{,0}(0,T)$, as extension of the inner product $\langle \cdot , \cdot \rangle_{L^2(0,T)}$ in $L^2(0,T).$ Last, we define the modified Hilbert transformation
${\mathcal{H}}_T\colon \, L^2(0,T) \to L^2(0,T)$ as
\begin{equation}\label{Einf:HT}
  ({\mathcal{H}}_Tv)(t) = \sum_{k=0}^\infty v_k
  \cos \left( \left( \frac{\pi}{2} + k\pi \right) \frac{t}{T} \right), \quad t \in (0,T),
\end{equation}
where
\begin{equation*}
  v_k = \frac{2}{T} \int_0^T v(t) \sin \left( \left( \frac{\pi}{2} + k\pi \right) \frac{t}{T} \right) \mathrm dt 
\end{equation*}
are the Fourier coefficients of the series representation
\begin{equation*}
  v(t) = \sum_{k=0}^\infty v_k
  \sin \left( \left( \frac{\pi}{2} + k\pi \right) \frac{t}{T} \right) \mathrm dt, 
  \quad t \in (0,T),
\end{equation*}
when $v \in L^2(0,T)$ is given. The modified Hilbert transformation
${\mathcal{H}}_T$ is a bijective mapping ${\mathcal{H}}_T\colon \, H^\nu_{0,}(0,T) \to H^\nu_{\,,0}(0,T)$ for $\nu \in \{0, 1/2, 1\}$ with $H^0_{0,}(0,T)=H^0_{,0}(0,T)=L^2(0,T)$. Additional properties of $\mathcal H_T$ are given in \cite{SteinbachZankETNA2020, SteinbachZankJNUM2021, ZankCMAM2021, SteinbachMissoni2022}.

With this notation, a related variational formulation to the parabolic-type problem~\eqref{Einf:Waerme} is to find 
$u \in H^{1/2}_{0,}(0,T)$ such that
\begin{equation} \label{Einf:WaermeVF}
  \forall v \in H^{1/2}_{0,}(0,T): \quad   \langle \partial_t u , {\mathcal{H}}_T v \rangle_{(0,T)} + \mu \langle u , {\mathcal{H}}_T v \rangle_{L^2(0,T)} = \langle f , {\mathcal{H}}_T v \rangle_{(0,T)}
\end{equation}
for a given right-hand side $f \in [H^{1/2}_{,0}(0,T)]'$. For the hyperbolic-type problem~\eqref{Einf:Welle}, a related variational formulation is to find 
$u \in H^1_{0,}(0,T)$ such that
\begin{equation} \label{Einf:WelleVF}
  \forall v \in H^1_{0,}(0,T): \quad   \langle {\mathcal{H}}_T \partial_t u , \partial_t v \rangle_{L^2(0,T)} + \mu \langle u , {\mathcal{H}}_T v \rangle_{L^2(0,T)} = \langle f , {\mathcal{H}}_T v \rangle_{(0,T)}
\end{equation}
for a given right-hand side $f \in [H^1_{,0}(0,T)]'$. The variational formulations~\eqref{Einf:WaermeVF}, \eqref{Einf:WelleVF} are uniquely solvable, where their derivation and analysis is given in \cite{SteinbachZankETNA2020, ZankDissBuch2020, LoescherSteinbachZankDD}.

Next, we state conforming discretizations of the variational formulations~\eqref{Einf:WaermeVF}, \eqref{Einf:WelleVF}. For this purpose, we define a partition ${\mathcal T}^N= \{\tau_\ell\}_{\ell=1}^N$ of $(0,T)$ for given $N\in \N$, i.e.,
\begin{equation} \label{Einf:Netz}
    0 = t_0 < t_1 < t_2 < \dots < t_{N -1} < t_{N} = T
\end{equation}
with elements $\tau_\ell := (t_{\ell-1},t_\ell) \subset \R$. The mesh sizes are $h_\ell := t_\ell - t_{\ell-1}$ for $\ell=1,\dots,N$, and the maximal mesh size is $h:=\max_\ell h_\ell$. On ${\mathcal T}^N$, with the distribution $\boldsymbol{p} = (p_1,\dots,p_N) \in \N^N$ of polynomial degrees, we introduce the space of piecewise polynomial, continuous functions on intervals
\begin{equation}\label{Einf:DefStwPolynome}
  S^{\boldsymbol{p}}({\mathcal T}^N) := \{ v \in C[0,T]: \, \forall \ell \in \{1,\dots,N\} : v_{|{\tau_\ell}} \in \mathbb{P}^{p_\ell}(\tau_\ell) \} = \mathrm{span} \{\varphi_i \}_{i=1}^M,
\end{equation}
where $C[0,T]$ is the space of continuous functions and $\mathbb P^p(A)$ is the space of polynomials on a subset $A \subset \R^d$, $d \in \N$, of global degree at most $p \in \N$. Here, $M = \sum_{\ell=1}^N (p_\ell + 1) - (N -1) = 1 + \sum_{\ell=1}^N p_\ell$ is the number of degrees of freedom and the functions $\varphi_i \colon \, [0,T] \to \R$ are basis functions of $S^{\boldsymbol{p}}({\mathcal T}^N)$, satisfying
\begin{equation} \label{Einf:BasisVarphi}
  \forall j \in \{ 2, \dots, M\} : \quad \varphi_j(0) = 0,
\end{equation}
i.e., $\varphi_1(0) \neq 0.$ Further, we define the subspace
\begin{equation*}
  S^{\boldsymbol{p}}_{0,}({\mathcal T}^N) := S^{\boldsymbol{p}}({\mathcal T}^N) \cap H^1_{0,}(0,T) = \mathrm{span}\{\varphi_{j+1}\}_{j=1}^{M-1}.
\end{equation*}
Thus, a conforming finite element method of the variational formulation~\eqref{Einf:WaermeVF} is to find $u_h \in S^{\boldsymbol{p}}_{0,}({\mathcal T}^N)$ such that
\begin{equation} \label{Einf:WaermeVFDisk}
   \forall v_h \in S^{\boldsymbol{p}}_{0,}({\mathcal T}^N): \quad  \langle \partial_t u_h , {\mathcal{H}}_T v_h \rangle_{L^2(0,T)} + \mu \langle u_h , {\mathcal{H}}_T v_h \rangle_{L^2(0,T)} = \langle f , {\mathcal{H}}_T v_h \rangle_{(0,T)}.
\end{equation}
The discrete variational formulation~\eqref{Einf:WaermeVFDisk} is uniquely solvable, see \cite{SteinbachZankETNA2020, ZankDissBuch2020}, and is equivalent to the linear system
\begin{equation} \label{Einf:LGSWaerme}
    (\widetilde A^{\mathcal H_T} + \mu \widetilde M^{\mathcal H_T}) \vec u = \vec{F}^{\mathcal H_T}
\end{equation}
with the matrices
\begin{align}
    \widetilde M^{\mathcal H_T}[i,j] :=& \langle \varphi_{j+1}, \mathcal H_T \varphi_{i+1} \rangle_{L^2(0,T)}, \label{Einf:MhTilde} \\
    \widetilde A^{\mathcal H_T}[i,j] :=& \langle \partial_t \varphi_{j+1}, \mathcal H_T \varphi_{i+1} \rangle_{L^2(0,T)}  \label{Einf:AhTilde}
\end{align}
for $i,j=1,\dots,M-1,$ and the corresponding right-hand side $\vec{F}^{\mathcal H_T}$.

Analogously, the conforming finite element method of the variational formulation~\eqref{Einf:WelleVF} to find $u_h \in S^{\boldsymbol{p}}_{0,}({\mathcal T}^N)$ such that
\begin{equation*}
   \forall v_h \in S^{\boldsymbol{p}}_{0,}({\mathcal T}^N): \quad  \langle {\mathcal{H}}_T\partial_t u_h , \partial_t  v_h \rangle_{L^2(0,T)} + \mu \langle u_h , {\mathcal{H}}_T v_h \rangle_{L^2(0,T)} = \langle f , {\mathcal{H}}_T v_h \rangle_{(0,T)}
\end{equation*}
is uniquely solvable, see \cite{LoescherSteinbachZankDD, LoescherSteinbachZankWelleTheorie}, and is equivalent to the linear system
\begin{equation*}
    ((\widetilde B^{\mathcal H_T})^\top + \mu \widetilde M^{\mathcal H_T}) \vec u = \vec{F}^{\mathcal H_T}
\end{equation*}
with the matrix $\widetilde M^{\mathcal H_T}$, given in \eqref{Einf:MhTilde}, and the matrix
\begin{equation}  \label{Einf:BhTilde}
    \widetilde B^{\mathcal H_T}[i,j] := \langle \partial_t \varphi_{j+1}, \mathcal H_T \partial_t \varphi_{i+1} \rangle_{L^2(0,T)}
\end{equation}
for $i,j=1,\dots,M-1,$ and the right-hand side $\vec{F}^{\mathcal H_T}$ as for the linear system~\eqref{Einf:LGSWaerme}. Note that for $f \in L^2(0,T)$, this right-hand side $\vec{F}^{\mathcal H_T}$ can be realized, using the matrix
\begin{equation}  \label{Einf:Mh}
    M^{\mathcal H_T}[i,j] := \langle \varphi_j, \mathcal H_T \varphi_i \rangle_{L^2(0,T)}
\end{equation}
for $i,j=1,\dots,M$, see \cite{ZankCMAM2021, ZankWaermeLoeserICOSAHOM2022} for all the details. Additionally, we set
\begin{align}
    A^{\mathcal H_T}[i,j] :=& \langle \partial_t \varphi_j, \mathcal H_T \varphi_i \rangle_{L^2(0,T)}, \label{Einf:Ah} \\
    B^{\mathcal H_T}[i,j] :=& \langle \partial_t \varphi_j, \mathcal H_T \partial_t \varphi_i \rangle_{L^2(0,T)}  \label{Einf:Bh}
\end{align}
for $i,j=1,\dots,M$. Hence, the matrices $\widetilde M^{\mathcal H_T}$, $\widetilde A^{\mathcal H_T}$, $\widetilde B^{\mathcal H_T}$, given in \eqref{Einf:MhTilde}, \eqref{Einf:AhTilde}, \eqref{Einf:BhTilde}, are submatrices of the matrices $M^{\mathcal H_T}$, $A^{\mathcal H_T}$, $B^{\mathcal H_T}$, given in \eqref{Einf:Mh}, \eqref{Einf:Ah}, \eqref{Einf:Bh}, respectively. Thus, we ask for an accurate realization of these matrices $M^{\mathcal H_T}$, $A^{\mathcal H_T}$, $B^{\mathcal H_T}$, which is the main topic of this manuscript.

The rest of the paper is organized as follows: In Section~\ref{sec:HT}, we recall well-known and prove new weakly singular integral representations of $\mathcal H_T$, which are the starting point of the calculation of the matrices $M^{\mathcal H_T}$, $A^{\mathcal H_T}$, $B^{\mathcal H_T}$. In Section~\ref{sec:Int}, Gauß quadratures are given, which are used in the following sections. Section~\ref{sec:NumHT} is the main part of this paper, where we state all quadrature schemes to calculate the matrices $M^{\mathcal H_T}$, $A^{\mathcal H_T}$, $B^{\mathcal H_T}$ to high accuracy. In Section~\ref{sec:Num}, numerical examples show the quality of the new assembling method of $M^{\mathcal H_T}$, $A^{\mathcal H_T}$, $B^{\mathcal H_T}$. In Section~\ref{sec:Zum}, we give some conclusions.

\section{Integral Representations of the Modified Hilbert Transformation} \label{sec:HT}

In this section, we recall a weakly singular integral representation of $\mathcal H_T$ for functions in $H^1(0,T)$ and prove a new weakly singular integral representation of $\mathcal H_T$ for functions, which are only piecewise in $H^1$. These weakly singular integral representations are used for the calculation of the matrices $M^{\mathcal H_T}$, $A^{\mathcal H_T}$, $B^{\mathcal H_T}$, given in \eqref{Einf:Mh}, \eqref{Einf:Ah}, \eqref{Einf:Bh}, respectively. We start to
recall the integral representations for functions in $L^2(0,T)$ or $H^1(0,T)$, which is proven in \cite{SteinbachZankJNUM2021}:

\begin{lemma}[Lemma~2.1 in \cite{SteinbachZankJNUM2021}] \label{HT:Lem:L2H1}
 For $v \in L^2(0,T),$ the operator ${\mathcal{H}}_T$, as defined in 
  \eqref{Einf:HT}, allows the integral representation
  \begin{equation}\label{HT:Hilbert Transformation Integral_L2}
  ({\mathcal{H}}_T v)(t) = {\mathrm  {v.p.}} \int_0^T  v(s) \, K(s,t) \, \mathrm ds, \quad t \in (0,T),
  \end{equation}
  as a Cauchy principal value integral, where the kernel function is given as
  \begin{equation*}
    K(s,t) := \frac{1}{2T} \left[  \frac{1}{\sin \frac{\pi(s+t)}{2T}  }  + \frac{1}{\sin \frac{\pi(s-t)}{2T}  }   \right].
  \end{equation*}
 For $v \in H^1(0,T),$ the operator ${\mathcal{H}}_T$, as defined in 
  \eqref{Einf:HT}, allows the integral representation
    \begin{equation}\label{HT:Hilbert Transformation Integral_H1}
    ({\mathcal{H}}_Tv)(t) =
    - \frac{2}{\pi} v(0) \ln \tan \frac{\pi t}{4T} + \int_0^{T} \partial_t v(s) \mathcal K(s,t) \mathrm ds, \quad t \in (0,T),
    \end{equation}
  as a weakly singular integral, where
  \begin{equation} \label{HT:Kern}
      \mathcal K(s,t) := -\frac{1}{\pi} \ln \left[ \tan \frac{\pi (s+t)}{4T} \tan \frac{\pi \abs{t-s}}{4T} \right].
  \end{equation}
\end{lemma}

The integral representation~\eqref{HT:Hilbert Transformation Integral_H1} in connection with specialized numerical integration gives the possibility to calculate matrices, which are based on $\mathcal H_T$ applied to a function in $H^1(0,T)$, e.g., the matrices $M^{\mathcal H_T}$ in \eqref{Einf:Mh} and $A^{\mathcal H_T}$ in \eqref{Einf:Ah}. However, the integral representation~\eqref{HT:Hilbert Transformation Integral_H1} is not applicable for functions, which have jumps, since these functions do not belong to $H^1(0,T)$. As the integral representation~\eqref{HT:Hilbert Transformation Integral_L2} is a Cauchy principal value integral and therefore, is complicated to use in practical implementations, we prove a new integral representation of $\mathcal H_T$ for functions, which are piecewise in $H^1$ but possibly having jumps. This integral representation is used for the calculation of the matrix $B^{\mathcal H_T}$ in \eqref{Einf:Bh}.

\begin{lemma} \label{HT:Lem:Darstellung}
  Let $a,b \in [0,T]$ with $a < b$ and a function $f \in H^1(a,b) \subset C[a,b]$ be given. Then, for the function $v_f \in L^2(0,T),$
  \begin{equation*}
    v_f(s) = \begin{cases}
                f(s), & s \in [a,b], \\
                0, & \text{otherwise,}
             \end{cases}
  \end{equation*}
  the operator ${\mathcal{H}}_T$, as defined in \eqref{Einf:HT}, allows the integral representation
    \begin{equation}\label{HT:Darstellung_HT}
    ({\mathcal{H}}_Tv_f)(t) = -f(b) \mathcal K(b,t) + f(a) \mathcal K(a,t) + \int_a^b \partial_t f(s) \mathcal K(s,t) \mathrm ds
    \end{equation}
  for $t \in (0,T) \setminus \{a, b\}$ as a weakly singular integral, where the kernel function $\mathcal K(s,t)$ is given in \eqref{HT:Kern}.
  
  If, in addition, $f(a) = 0$, i.e., $f \in H^1_{0,}(a,b)$, then the representation
  \begin{equation*}
   ({\mathcal{H}}_Tv_f)(a) = -f(b) \mathcal K(b,a) + \int_a^b \partial_t f(s) \mathcal K(s,a) \mathrm ds
  \end{equation*}
  holds true as a weakly singular integral. Analogously, if $f(b) = 0$, i.e., $f \in H^1_{,0}(a,b)$, then we have the representation
  \begin{equation*}
   ({\mathcal{H}}_Tv_f)(b) = f(a) \mathcal K(a,b) + \int_a^b \partial_t f(s) \mathcal K(s,b) \mathrm ds
  \end{equation*}
  as a weakly singular integral.
\end{lemma}
\begin{proof}
    Let $t \in (0,T)$ be arbitrary but fixed. We split the proof accordingly to the relation between $t$, $a$ and $b$.
    
    First, consider the case $t \in (a,b)$. The integral representation~\eqref{HT:Hilbert Transformation Integral_L2}, integration by parts, the relation $-\partial_s \mathcal K(s,t) = K(s,t)$ and the 
    Hölder continuity of $f$ with exponent $\frac 1 2$, see \cite[Chapitre~2, Théorème~3.8] {Necas1967}, yield
    \begin{align*}
      ({\mathcal{H}}_Tv_f)(t) &= \lim_{\varepsilon \searrow 0} \left( \int_a^{t-\varepsilon} f(s) K(s,t) \mathrm ds +  \int_{t+\varepsilon}^b f(s) K(s,t) \mathrm ds \right) \\
      &= \lim_{\varepsilon \searrow 0} \left( -f(t-\varepsilon) \mathcal K(t-\varepsilon,t) + f(a) \mathcal K(a,t) +  \int_a^{t-\varepsilon} \partial_t f(s) \mathcal K(s,t) \mathrm ds \right. \\
      &\qquad -\left. f(b) \mathcal K(b,t) + f(t+\varepsilon) \mathcal K(t+\varepsilon,t) + \int_{t+\varepsilon}^b \partial_t f(s) \mathcal K(s,t) \mathrm ds \right) \\
      &= f(a) \mathcal K(a,t) - f(b) \mathcal K(b,t)  + \int_a^b \partial_t f(s) \mathcal K(s,t) \mathrm ds,
    \end{align*}
    where the last integral exists as a weakly singular integral.
    
    Second, we examine the case $t \in (0,a)$ when $a>0$, or $t \in (b,T)$ when $b<T$. The integral representation~\eqref{HT:Hilbert Transformation Integral_L2} and integration by parts give
    \begin{equation*}
      ({\mathcal{H}}_Tv_f)(t) = \lim_{\varepsilon \searrow 0} \int_a^b f(s) K(s,t) \mathrm ds = f(a) \mathcal K(a,t) - f(b) \mathcal K(b,t)  + \int_a^b \partial_t f(s) \mathcal K(s,t) \mathrm ds.
    \end{equation*}
    
    Third, the case $t=a$ is investigated when $f(a) = 0$. As in the cases before, the integral representation~\eqref{HT:Hilbert Transformation Integral_L2}, the Hölder continuity of $f$ and integration by parts lead to
    \begin{align*}
      ({\mathcal{H}}_Tv_f)(a) &= \lim_{\varepsilon \searrow 0} \int_{a+\varepsilon}^b f(s) K(s,a) \mathrm ds \\
      &= \lim_{\varepsilon \searrow 0} \left( f(a+\varepsilon) \mathcal K(a+\varepsilon,a) - f(b) \mathcal K(b,a)  + \int_a^b \partial_t f(s) \mathcal K(s,a) \mathrm ds \right) \\
      &= - f(b) \mathcal K(b,a)  + \int_a^b \partial_t f(s) \mathcal K(s,a) \mathrm ds,
    \end{align*}
    where the last integral exists as a weakly singular integral.
    
    Last, the case $t=b$, when $f(b)=0$, is proven analogously.
\end{proof}

\begin{remark}
 Choose $a=0$ and $b=T$ in Lemma~\ref{HT:Lem:Darstellung}. Then, the representation~\eqref{HT:Darstellung_HT} of Lemma~\ref{HT:Lem:Darstellung} coincides with the representation~\eqref{HT:Hilbert Transformation Integral_H1} of Lemma~\ref{HT:Lem:L2H1} due to $\mathcal K(0,t) = -\frac{2}{\pi} \ln \tan \frac{\pi t}{4T}$ and $\mathcal K(T,t) = 0$ for $t \in (0,T).$
\end{remark}

\begin{remark}
 Consider the situation of Lemma~\ref{HT:Lem:Darstellung}. For $f \in H^1(a,b)$ with $f(a) \neq 0$ or $f(b)\neq 0$, the function $\mathcal H_T v_f \in L^2(0,T)$ admits a singularity for $t=a$ or $t=b$, respectively.
\end{remark}

\section{Gauß Quadratures} \label{sec:Int}

In this section, Gauß quadratures are stated, which are used in the following sections. For this purpose, we introduce Gauß quadratures on $[0,1]$ of order $K \in \N$, see \cite{Gautschi2012Log}. First, the classical Gauß--Legendre quadrature with weights $\omega_{\nu,K} \in \R$ and nodes $\xi_{\nu,K} \in [0,1]$, $\nu=1,\dots,K$, fulfills
\begin{equation} \label{Int:GaussLegendre}
  \forall g \in \mathbb{P}^{2K-1}[0,1] : \quad  \int_0^1 g(t) \mathrm dt = \sum_{\nu=1}^K \omega_{\nu,K} g(\xi_{\nu,K}),
\end{equation}
see \cite[(1.6)]{Gautschi2012Log}. 
This Gauß quadrature is generalized to the square domain $[0,1] \times [0,1]$ via using the Gauß--Legendre quadrature~\eqref{Int:GaussLegendre} for each coordinate direction. More precisely, the tensor Gauß quadrature 
\begin{equation} \label{Int:GaussTensor}
 \sum_{\nu_1=1}^{K_1}  \sum_{\nu_2=1}^{K_2} \omega_{\nu_1,K_1} \omega_{\nu_2,K_2} G(\xi_{\nu_1,K_1}, \xi_{\nu_2,K_2}),
\end{equation}
approximates the integral
\begin{equation*}
  \int_0^1 \int_0^1 G(\xi_1,\xi_2) \, \mathrm d \xi_1 \mathrm d \xi_2
\end{equation*}
for a function $G\colon \, [0,1] \times [0,1] \to \R$, where $\omega_{\nu_i,K_i} \in \R$ and $\xi_{\nu_i,K_i} \in [0,1]$ are the Gauß integration weights and Gauß integration nodes of order $K_i \in \mathbb N$  with respect to the coordinate direction $i \in \{1,2\}$.

For singular integrals with a logarithmic term, we introduce a nonclassical Gauß--Jacobi quadrature. In greater detail, for $K \in \N$, the weights $\hat \omega_{\nu,K} \in \R$ and nodes $\hat \xi_{\nu,K} \in [0,1]$, $\nu=1,\dots,K$, of the Gauß--Jacobi quadrature for a logarithmic term satisfy
\begin{equation} \label{Int:Gausslog}
  \forall g \in \mathbb{P}^{2K-1}[0,1] : \quad  -\int_0^1 g(t) \ln(t) \mathrm dt = \sum_{\nu=1}^K \hat \omega_{\nu,K} g(\hat \xi_{\nu,K}),
\end{equation}
see \cite[(1.7)]{Gautschi2012Log}.
With the quadrature \eqref{Int:Gausslog}, the equality
\begin{multline} \label{Int:GaussLogTensor}
    \int_0^1 \int_0^1 g(s,t) \ln \abs{s-t} \mathrm ds \mathrm dt = \\
    -\sum_{\mu=1}^{K}\sum_{\nu=1}^{K} \hat \omega_{\mu}  \hat \xi_{\mu} \omega_{\nu} \left(  g((1-\xi_{\nu})\hat \xi_{\mu}, \hat \xi_{\mu}) +  g(1-(1-\xi_{\nu})\hat \xi_{\mu},1-\hat \xi_{\mu})  \right) \\
    - \sum_{\mu=1}^{K}\sum_{\nu=1}^{K} \omega_{\mu} \xi_{\mu} \hat \omega_{\nu} \left(  g((1-\hat \xi_{\nu})\xi_{\mu}, \xi_{\mu}) +  g(1-(1-\hat \xi_{\nu}) \xi_{\mu},1- \xi_{\mu})  \right)
\end{multline}
holds true for all $g \in \mathbb{P}^{2K-2}([0,1]\times[0,1])$, where
\begin{equation*}
  \omega_\nu = \omega_{\nu,K}, \quad \hat \omega_\nu = \hat \omega_{\nu,K}, \quad \xi_\nu = \xi_{\nu,K}, \quad \hat \xi_\nu = \hat \xi_{\nu,K}
\end{equation*}
are defined in \eqref{Int:GaussLegendre}, \eqref{Int:Gausslog}, see \cite[(2.7)]{Gautschi2012Log}, \cite{Gautschi2012LogErratum}.

\section{Quadrature Schemes for the Modified Hilbert Transformation} \label{sec:NumHT}

In this section, we describe the assembling of the matrices $M^{\mathcal H_T}$ in \eqref{Einf:Mh}, $A^{\mathcal H_T}$ in \eqref{Einf:Ah} and $B^{\mathcal H_T}$ in \eqref{Einf:Bh}. The crucial
point is the realization of the modified Hilbert transformation $\mathcal H_T$,
where different possibilities exist, see
\cite{SteinbachZankJNUM2021, ZankCMAM2021}. 
In particular, for
a uniform degree vector $\boldsymbol{p} = (p, p, \dots, p)$ with a fixed, low
polynomial degree $p \in \N$, e.g., $p=1$ or $p=2$, the matrices $M^{\mathcal H_T}$, $A^{\mathcal H_T}$ and $B^{\mathcal H_T}$ in \eqref{Einf:Mh}, \eqref{Einf:Ah}, \eqref{Einf:Bh}, respectively, can be calculated using a
series expansion based on the \textit{Legendre chi function}, which
converges very fast, independently of the mesh sizes, see
\cite[Subsection~2.2]{ZankCMAM2021}. As for an
$hp$-FEM, the degree vector $\boldsymbol{p}$ is not uniform or for high polynomial degrees, it is convenient to apply numerical quadrature rules to numerically approximate the matrix entries of $M^{\mathcal H_T}$, $A^{\mathcal H_T}$ and $B^{\mathcal H_T}$ in \eqref{Einf:Mh}, \eqref{Einf:Ah}, \eqref{Einf:Bh}, respectively. This is the main topic of the paper and is described in great detail in the following.

From the integral representation of $\mathcal H_T$ in \eqref{HT:Hilbert Transformation Integral_H1}, we have
\begin{multline}
    M^{\mathcal H_T}[i, j] =
    \langle \varphi_j , \mathcal H_T \varphi_i
    \rangle_{L^2(0,T)}  = - \frac{2}{\pi} \varphi_i(0) \int_0^T \varphi_j(t) \ln \tan \frac{\pi t}{4T} \mathrm dt \\
    +  \int_0^T \varphi_j(t)
    \int_0^T \mathcal K(s,t) \, \partial_t \varphi_i(s) \, \mathrm ds \, \mathrm dt   \label{NumHT:Mht_entries}
\end{multline}
and
\begin{multline}
    A^{\mathcal H_T}[i, j] =
    \langle \partial_t \varphi_j , \mathcal H_T \varphi_i
    \rangle_{L^2(0,T)}  
    = - \frac{2}{\pi} \varphi_i(0) \int_0^T \partial_t \varphi_j(t) \ln \tan \frac{\pi t}{4T} \mathrm dt \\
    +  \int_0^T \partial_t \varphi_j(t)
    \int_0^T \mathcal K(s,t) \, \partial_t \varphi_i(s) \, \mathrm ds \, \mathrm dt   \label{NumHT:Aht_entries}
\end{multline}
for $i,j=1,\ldots,M$ with the basis functions $\varphi_i, \varphi_j$ in \eqref{Einf:DefStwPolynome} satisfying \eqref{Einf:BasisVarphi}.

For the matrix $B^{\mathcal H_T}$ in \eqref{Einf:Bh}, the integral representation of $\mathcal H_T$ in \eqref{HT:Darstellung_HT} yields
\begin{equation*}
    \mathcal H_T \partial_t \varphi_i(t) = \sum_{k=1}^N \Big[ - (\partial_t \varphi_i)^k_- \mathcal K(t_k,t) +  (\partial_t \varphi_i)^{k-1}_+ \mathcal K(t_{k-1},t)  + \int_{t_{k-1}}^{t_k}  \partial_{tt} (\varphi_{i|\tau_k})(s) \, \mathcal K(s,t) \, \mathrm ds \Big]
\end{equation*}
for $t \in (0,T) \setminus \{ t_k : k=1,\dots,N-1 \}$, $i=1,\dots,M,$ where we use the notation
\begin{equation*}
    v^{k}_-:=\lim_{\varepsilon \searrow 0} v (t_k-\varepsilon) \quad \text{ and } \quad v^{k-1}_+:=\lim_{\varepsilon \searrow 0} v (t_{k-1}+\varepsilon), \quad k=1,\dots,N,
\end{equation*}
for a sufficiently smooth function $v\colon \, (0,T) \to \R.$ Thus, the entries of the matrix in \eqref{Einf:Bh} admit the representation
\begin{align}
     B^{\mathcal H_T}[i, j] =& \langle \partial_t \varphi_j , \mathcal H_T \partial_t \varphi_i
    \rangle_{L^2(0,T)} \nonumber \\ 
    =& \sum_{k=1}^N \Big[ - (\partial_t \varphi_i)^k_- \int_0^T \partial_t \varphi_j(t) \mathcal K(t_k,t) \mathrm dt  +  (\partial_t \varphi_i)^{k-1}_+ \int_0^T \partial_t \varphi_j(t) \mathcal K(t_{k-1},t) \mathrm dt \nonumber \\
    &\qquad + \int_0^T \partial_t \varphi_j(t) \int_{t_{k-1}}^{t_k}  \partial_{tt} (\varphi_{i|\tau_k})(s) \, \mathcal K(s,t) \, \mathrm ds \mathrm dt \Big]    \label{NumHT:Bht_entries}
\end{align}
for $i,j=1,\ldots,M$ with the basis functions $\varphi_i, \varphi_j$ in \eqref{Einf:DefStwPolynome} satisfying \eqref{Einf:BasisVarphi}.

The matrix entries $M^{\mathcal H_T}[i, j]$, $A^{\mathcal H_T}[i, j]$, $B^{\mathcal H_T}[i, j]$  in \eqref{NumHT:Mht_entries}, \eqref{NumHT:Aht_entries}, \eqref{NumHT:Bht_entries}, respectively, are computed element-wise for the partition $\mathcal T^N = \{\tau_\ell\}_{\ell=1}^N$ of $(0,T)$ into intervals $\tau_\ell = (t_{\ell-1},t_\ell)\subset (0,T)$, $\ell=1,\dots, N$. For this purpose, we assume that we use standard Lagrange finite elements as in \cite[Section~6.3]{ErnGuermondFEMI2021}, or for the $hp$-FEM, we use basis functions $\varphi_i$, $i=1,\dots,M$, as described in \cite[Subsection~3.1.6]{Schwabphp1998}. Thus, we assume that the basis functions $\varphi_i \colon \, [0,T] \to \R$, $i=1,\dots,M$, allow for a local representation on an element $\tau_\ell$, $\ell=1,\dots,N,$ by shape functions $\psi_m^{p_\ell} \colon \, [0,1] \to \R$, $m=1,\dots,p_\ell+1$, defined on the reference interval $[0,1].$ In greater detail, for a basis function $\varphi_i$, $i=1,\dots,M$, there exists an element $\tau_\ell = (t_{\ell-1}, t_\ell)$, $\ell=1,\dots,N$, with a related polynomial degree $p_\ell \in \N$ such that the local representation
\begin{equation} \label{NumHT:Formfunktionen}
    \forall t \in \overline{\tau}_\ell : \quad \varphi_{\alpha(m,\ell)}(t) = \psi_m^{p_\ell} \left( \frac{t-t_{\ell-1}}{h_\ell} \right)
\end{equation}
holds true, where $\alpha(m,\ell)=i \in \{1,2,\dots,M\}$ is the global index related to the local index $m$ for the interval $\tau_\ell$ and $\psi_m^{p_\ell}$ is a shape function. Possible choices of such shape functions are the classical Lagrange polynomials with equidistant nodes or Gauß--Lobatto nodes, see \cite[Chapter~6]{ErnGuermondFEMI2021}, or Lobatto polynomials (integrated Legendre polynomials), see \cite[Subsection~3.1.4]{Schwabphp1998}.     

With this notation, we fix two intervals $\tau_\ell = (t_{\ell-1},t_\ell)$, $\tau_k = (t_{k-1},t_k)$ with indices $\ell, k \in \{1,\dots,N\}$ and related local polynomial degrees $p_\ell$, $p_k \in \N$. We define the local matrix $M^{\mathcal H_T}_{k,\ell} \in \R^{(p_k+1) \times (p_\ell+1)}$ by
\begin{align} 
 M^{\mathcal H_T}_{k,\ell}[n, m] =& - \frac{2}{\pi} \varphi_{\alpha(n,k)}(0) \int_{t_{\ell-1}}^{t_\ell} \varphi_{\alpha(m,\ell)}(t) \ln \tan \frac{\pi t}{4T} \mathrm dt \nonumber \\
 &+ \int_{t_{\ell-1}}^{t_\ell} \varphi_{\alpha(m,\ell)}(t) 
  \int_{t_{k-1}}^{t_k} \mathcal K(s,t) \partial_t \varphi_{\alpha(n,k)}(s) \, \mathrm ds \, \mathrm dt   \nonumber \\
  =& - \delta_{k,1} \psi^{p_k}_n(0) \frac{2 h_\ell}{\pi}  \int_0^1 \psi^{p_\ell}_m(\eta) \ln \tan \frac{\pi (t_{\ell-1}+\eta h_\ell)}{4T} \mathrm d\eta  \nonumber \\
  &+ h_\ell \int_0^1 \psi^{p_\ell}_m(\eta) \int_0^1 \mathcal K(t_{k-1}+\xi h_k,t_{\ell-1}+\eta h_\ell) \partial_t \psi^{p_k}_n(\xi) \, \mathrm d\xi \, \mathrm d\eta   \label{NumHT:Mht_local}
\end{align}
and, analogously, the local matrix $A^{\mathcal H_T}_{k,\ell} \in \R^{(p_k+1) \times (p_\ell+1)}$ by
\begin{multline} 
 A^{\mathcal H_T}_{k,\ell}[n, m] = - \delta_{k,1} \psi^{p_k}_n(0) \frac{2}{\pi}  \int_0^1 \partial_t \psi^{p_\ell}_m(\eta) \ln \tan \frac{\pi (t_{\ell-1}+\eta h_\ell)}{4T} \mathrm d\eta  \\
  +  \int_0^1 \partial_t \psi^{p_\ell}_m(\eta) \int_0^1 \mathcal K(t_{k-1}+\xi h_k,t_{\ell-1}+\eta h_\ell) \partial_t \psi^{p_k}_n(\xi) \, \mathrm d\xi \, \mathrm d\eta   \label{NumHT:Aht_local}
\end{multline}
for $n=1,\dots,p_k+1$ and $m=1,\dots,p_\ell+1$, where we used the property~\eqref{Einf:BasisVarphi}, the local representation~\eqref{NumHT:Formfunktionen} and $\delta_{k,1}$ is the usual Kronecker delta.

In the same way, the local matrix $B^{\mathcal H_T}_{k,\ell} \in \R^{(p_k+1) \times (p_\ell+1)}$ is defined by
\begin{multline} \label{NumHT:Bht_local}
 B^{\mathcal H_T}_{k,\ell}[n, m] = \frac{-1}{\pi h_k} \partial_t \psi^{p_k}_n(0) J_{k,\ell}^0[m] + \frac{1}{\pi h_k} \partial_t \psi^{p_k}_n(1) J_{k,\ell}^1[m] \\
   + \frac{1}{h_k} \int_0^1 \partial_t \psi^{p_\ell}_m(\eta) \int_0^1 \mathcal K(t_{k-1}+\xi h_k,t_{\ell-1}+\eta h_\ell) \partial_{tt} \psi^{p_k}_n(\xi) \, \mathrm d\xi \, \mathrm d\eta 
\end{multline}
for $n=1,\dots,p_k+1$ and $m=1,\dots,p_\ell+1$ with the integrals
\begin{align}
    J_{k,\ell}^0[m] &= -\pi \int_0^1 \partial_t \psi^{p_\ell}_m(\eta) \mathcal K(t_{k-1}, t_{\ell-1}+\eta h_\ell) \mathrm d\eta, \label{NumHT:Bht_local_J0}\\
    J_{k,\ell}^1[m] &=  -\pi \int_0^1 \partial_t \psi^{p_\ell}_m(\eta) \mathcal K(t_k, t_{\ell-1}+\eta h_\ell) \mathrm d\eta.  \label{NumHT:Bht_local_J1}
\end{align}
Note that $J_{k,\ell}^0[m]=J_{k-1,\ell}^1[m]$, i.e., investigating only one term is another possibility. To keep the same cases as for the matrices $M^{\mathcal H_T}$, $A^{\mathcal H_T}$ and therefore, due to implementation issues, we give the quadrature schemes for $J_{k,\ell}^0[m]$ and $J_{k,\ell}^1[m]$.

These local matrices  $M^{\mathcal H_T}_{k,\ell}$, $A^{\mathcal H_T}_{k,\ell}$, $B^{\mathcal H_T}_{k,\ell}$ are used for the calculation of the matrices $M^{\mathcal H_T}$, $A^{\mathcal H_T}$ and $B^{\mathcal H_T}$, respectively. Thus, the assembling of the matrices $M^{\mathcal H_T}$, $A^{\mathcal H_T}$ and $B^{\mathcal H_T}$ is realized element-wise, i.e., by two loops via the elements:
\begin{enumerate}
  \item Set $M^{\mathcal H_T} = 0_M$, $A^{\mathcal H_T} = 0_M$ and $B^{\mathcal H_T}=0_M$, where $0_M \in \R^{M \times M}$ is the zero matrix.
  \item For $k,\ell=1,\dots,N$,
    \begin{itemize}
      \item compute the matrices $M^{\mathcal H_T}_{k,\ell}, A^{\mathcal H_T}_{k,\ell}, B^{\mathcal H_T}_{k,\ell} \in \R^{(p_k+1) \times (p_\ell+1)}$,
      \item and set
            \begin{align*}
                M^{\mathcal H_T}[\alpha(n,k),\alpha(m,\ell)] &= M^{\mathcal H_T}[\alpha(n,k),\alpha(m,\ell)] + M^{\mathcal H_T}_{k,\ell}[n,m], \\
                A^{\mathcal H_T}[\alpha(n,k),\alpha(m,\ell)] &= A^{\mathcal H_T}[\alpha(n,k),\alpha(m,\ell)] + A^{\mathcal H_T}_{k,\ell}[n,m], \\
                B^{\mathcal H_T}[\alpha(n,k),\alpha(m,\ell)] &= B^{\mathcal H_T}[\alpha(n,k),\alpha(m,\ell)] + B^{\mathcal H_T}_{k,\ell}[n,m]
            \end{align*}
            for $n=1,\dots, p_k+1$, $m=1,\dots,p_\ell+1$.
    \end{itemize}
\end{enumerate}
It remains to calculate the local matrices $M^{\mathcal H_T}_{k,\ell}, A^{\mathcal H_T}_{k,\ell}, B^{\mathcal H_T}_{k,\ell} \in \R^{(p_k+1) \times (p_\ell+1)}$, which is the main purpose of this work and the content of the following subsections Subsection~\ref{sec:NumHTMh}, \ref{sec:NumHTAh}, \ref{sec:NumHTBh}, respectively.

\subsection{Local Matrix $M^{\mathcal H_T}_{k,\ell}$} \label{sec:NumHTMh}

In this subsection, the computation of the local matrix $M^{\mathcal H_T}_{k,\ell}$ in \eqref{NumHT:Mht_local} is investigated for fixed elements $\tau_k$, $\tau_\ell$ and local polynomial degrees $p_k$, $p_\ell$ with $k, \ell \in \{1,\dots,N\}$. We apply the strategy of \cite[Subsection~3.1]{SteinbachZankJNUM2021}, i.e., the integrals in \eqref{NumHT:Mht_local} are split into regular and singular parts. As in \cite[Subsection~3.1]{SteinbachZankJNUM2021}, we distinguish three different cases for the element indices $k$, $\ell$, which correspond to the singularities of the integrands in \eqref{NumHT:Mht_local}. For this purpose, let $n \in \{1,\dots,p_k+1\}$ and $m \in \{1,\dots,p_\ell+1\}$ be fixed. We investigate the case $k=\ell=1$ in great detail, whereas we state only the final results for the remaining cases, since they can be treated in a simpler or similar way as the case $k=\ell=1.$ These final results are formulated as integrals such that these integrals can be directly calculated or approximated by the quadratures~\eqref{Int:GaussLegendre}, \eqref{Int:GaussTensor}, \eqref{Int:Gausslog}, \eqref{Int:GaussLogTensor}.

\subsubsection{Case $k=\ell=1$} \label{sec:NumHTMh:Fall11}

We have
\begin{align*} 
 M^{\mathcal H_T}_{1,1}[n, m] =& - \psi^{p_1}_n(0) \frac{2 h_1}{\pi}  \int_0^1 \psi^{p_1}_m(\eta) F_{1,1}^1( \eta h_1 ) \mathrm d\eta - \psi^{p_1}_n(0) \frac{2 h_1}{\pi} \underbrace{ \int_0^1 \psi^{p_1}_m(\eta) \ln \eta \mathrm d\eta }_{=: I_{1,1}^1 } \\
  &- \psi^{p_1}_n(0) \frac{2 h_1}{\pi} \ln h_1 \int_0^1 \psi^{p_1}_m(\eta) \mathrm d\eta   - \frac{h_1}{\pi} \int_0^1 \psi^{p_1}_m(\eta) \int_0^1 F_{1,1}^2(\xi h_1,\eta h_1) \partial_t \psi^{p_1}_n(\xi) \, \mathrm d\xi \, \mathrm d\eta   \\
  &- \frac{h_1}{\pi} \underbrace{ \int_0^1 \psi^{p_1}_m(\eta) \int_0^1 \ln \abs{s-t}_{|s=\xi h_1, t=\eta h_1} \partial_t \psi^{p_1}_n(\xi) \, \mathrm d\xi \, \mathrm d\eta }_{ =: I_{1,1}^2 } \\
  &-\frac{h_1}{\pi} \underbrace{ \int_0^1 \psi^{p_1}_m(\eta) \int_0^1 \ln (s+t)_{|s=\xi h_1, t=\eta h_1} \partial_t \psi^{p_1}_n(\xi) \, \mathrm d\xi \, \mathrm d\eta }_{ =: I_{1,1}^3 }
\end{align*}
with
\begin{equation*}
  F_{1,1}^1(t):= \ln \frac { \tan \frac{\pi t}{4T} }{ t } ,\quad  F_{1,1}^2(s,t) := \ln \left[ \frac{ \tan \frac{\pi (s+t)}{4T} }{s+t}  \frac{ \tan \frac{\pi \abs{t-s}}{4T} }{ \abs{s-t}} \right].
\end{equation*}
Proceeding as in \cite[(3.10)]{SteinbachZankJNUM2021}, i.e., the square domain of integration is split into two triangles, and on 
each of these triangles a Duffy transformation is applied, yields
\begin{align*} 
 M^{\mathcal H_T}_{1,1}[n, m] =& - \psi^{p_1}_n(0) \frac{2 h_1}{\pi}  I_{1,1}^1  - \frac{h_1}{\pi}  \left( I_{1,1}^2 + I_{1,1}^3 \right) \\
  &- \psi^{p_1}_n(0) \frac{2 h_1}{\pi}  \int_0^1 \psi^{p_1}_m(\eta) F_{1,1}^1( \eta h_1 ) \mathrm d\eta  - \psi^{p_1}_n(0) \frac{2 h_1}{\pi} \ln h_1 \int_0^1 \psi^{p_1}_m(\eta) \mathrm d\eta  \\
  &- \frac{h_1}{\pi} \int_0^1 \int_0^1 \underbrace{ \psi^{p_1}_m(\eta) F_{1,1}^2((1-\xi)\eta h_1,\eta h_1) \eta \partial_t \psi^{p_1}_n((1-\xi)\eta) }_{ =: G_{1,1}^1(\xi,\eta) } \, \mathrm d\xi \, \mathrm d\eta \\
  &- \frac{h_1}{\pi} \int_0^1 \int_0^1 \underbrace{ \psi^{p_1}_m((1-\eta)\xi) F_{1,1}^2(\xi h_1,(1-\eta)\xi h_1) \xi \partial_t \psi^{p_1}_n(\xi) }_{ =: G_{1,1}^2(\xi,\eta) }  \, \mathrm d\eta \, \mathrm d\xi.
\end{align*}
Further, applying the Gauß--Legendre quadratures \eqref{Int:GaussLegendre}, \eqref{Int:GaussTensor} of order $K \in \N$ gives the approximation
\begin{align*} 
 M^{\mathcal H_T}_{1,1}[n, m] \approx& - \psi^{p_1}_n(0) \frac{2 h_1}{\pi}  I_{1,1}^1  - \frac{h_1}{\pi} \left( I_{1,1}^2 + I_{1,1}^3 \right)\\
  &- \psi^{p_1}_n(0) \frac{2 h_1}{\pi} \sum_{\nu=1}^K \omega_{\nu,K} \psi^{p_1}_m(\xi_{\nu,K}) \left( F_{1,1}^1( \xi_{\nu,K} h_1 ) + \ln h_1 \right)\\
  &- \frac{h_1}{\pi} \sum_{\nu_1=1}^K \sum_{\nu_2=1}^K \omega_{\nu_1,K} \omega_{\nu_2,K} \left( G_{1,1}^1(\xi_{\nu_1,K},\xi_{\nu_2,K}) + G_{1,1}^2(\xi_{\nu_1,K},\xi_{\nu_2,K}) \right).
\end{align*}
Using \cite[Corollary~3.1]{SteinbachZankJNUM2021} for the two-dimensional parts and similar arguments for the one-di\-men\-sion\-al part results in exponential convergence of the applied Gauß--Legendre quadratures with respect to the number $K$ of Gauß integration nodes. It remains to calculate the singular integrals $I_{1,1}^1$, $I_{1,1}^2$ and  $I_{1,1}^3$. In \cite[Subsection~3.1]{SteinbachZankJNUM2021}, an analytic integration is proposed, which is practical only for constant low polynomial degrees, e.g., $p_k=p_\ell \in \{1,2\}$. For arbitrary high polynomial degrees $p_k, p_\ell$, quadrature rules of order adapted to $p_k, p_\ell$ are easier to implement. This approach is not contained in \cite{SteinbachZankJNUM2021} and thus, is investigated in the following.

For the integral $I_{1,1}^1$, the Gauß--Jacobi quadrature \eqref{Int:Gausslog} for a logarithmic term yields
\begin{equation*}
  I_{1,1}^1 = \int_0^1 \psi^{p_1}_m(\eta) \ln \eta \mathrm d\eta = -\sum_{\nu=1}^{\Klog} \hat \omega_{\nu,\Klog} \psi^{p_1}_m(\hat \xi_{\nu,\Klog}) 
\end{equation*}
with $\frac{p_1+1}{2} \leq \Klog \in \N$.

For the integral $I_{1,1}^2$, we apply the Gauß--Legendre quadrature~\eqref{Int:GaussLegendre} and the quadrature~\eqref{Int:GaussLogTensor}, which lead to
\begin{align*}
    I_{1,1}^2 =& \ln h_1 \int_0^1 \psi^{p_1}_m(\eta) \int_0^1 \partial_t \psi^{p_1}_n(\xi) \, \mathrm d\xi \, \mathrm d\eta + \int_0^1 \int_0^1 \underbrace{ \psi^{p_1}_m(\eta) \partial_t \psi^{p_1}_n(\xi) }_{=H_{1,1}(\xi,\eta)} \ln \abs{\xi-\eta}\, \mathrm d\xi \, \mathrm d\eta \\
    =& \ln h_1 \left( \psi^{p_1}_n(1) - \psi^{p_1}_n(0) \right) \sum_{\nu=1}^{\Kreg} \omega_{\nu,\Kreg} \psi^{p_1}_m(\xi_{\nu,\Kreg}) \\
    &- \sum_{\mu=1}^{\Klog}\sum_{\nu=1}^{\Klog} \hat \omega_{\mu}  \hat \xi_{\mu} \omega_{\nu} \left(  H_{1,1}((1-\xi_{\nu})\hat \xi_{\mu}, \hat \xi_{\mu}) +  H_{1,1}(1-(1-\xi_{\nu})\hat \xi_{\mu},1-\hat \xi_{\mu})  \right) \\
    &- \sum_{\mu=1}^{\Klog}\sum_{\nu=1}^{\Klog} \omega_{\mu} \xi_{\mu} \hat \omega_{\nu} \left(  H_{1,1}((1-\hat \xi_{\nu})\xi_{\mu}, \xi_{\mu}) +  H_{1,1}(1-(1-\hat \xi_{\nu}) \xi_{\mu},1- \xi_{\mu})  \right)
\end{align*}
with $\frac{p_1+1}{2} \leq \Kreg \in \N$, $p_1 + \frac 1 2 \leq \Klog \in \N$ and the function
\begin{equation*}
H_{1,1}(\xi,\eta) := \psi^{p_1}_m(\eta) \partial_t \psi^{p_1}_n(\xi).
\end{equation*}

For the integral $I_{1,1}^3$, the square domain of integration of the first line is split into two triangles, and on 
each of these triangles a Duffy transformation is applied, which yields
\begin{align*}
    I_{1,1}^3 &= \ln h_1 \int_0^1 \psi^{p_1}_m(\eta) \int_0^1 \partial_t \psi^{p_1}_n(\xi) \, \mathrm d\xi \, \mathrm d\eta  +  \int_0^1 \int_0^1 \psi^{p_1}_m(\eta)  \ln (\xi+\eta) \partial_t \psi^{p_1}_n(\xi) \, \mathrm d\xi \, \mathrm d\eta \\
    &= \ln h_1 \int_0^1 \psi^{p_1}_m(\eta) \int_0^1 \partial_t \psi^{p_1}_n(\xi)  \mathrm d\xi \mathrm d\eta +  \int_0^1 \int_0^1 \psi^{p_1}_m(\eta) \eta   \underbrace{ \ln (\eta\xi+\eta)}_{=\ln \eta + \ln(1+\xi) } \partial_t \psi^{p_1}_n(\xi \eta) \mathrm d\xi \mathrm d\eta \\
    &\qquad+ \int_0^1 \int_0^1 \psi^{p_1}_m(\xi\eta) \xi \underbrace{ \ln (\xi+\xi\eta) }_{=\ln\xi + \ln(1+\eta)} \partial_t \psi^{p_1}_n(\xi) \, \mathrm d\xi \, \mathrm d\eta.
\end{align*}
Further, using the quadratures~\eqref{Int:GaussLegendre}, \eqref{Int:Gausslog}, \eqref{Int:GaussTensor} gives
\begin{multline*}
    I_{1,1}^3 \approx \ln h_1 \left( \psi^{p_1}_n(1) - \psi^{p_1}_n(0) \right) \sum_{\nu=1}^{K_1} \omega_{\nu,K_1} \psi^{p_1}_m(\xi_{\nu,K_1}) - \sum_{\mu=1}^{K_2} \sum_{\nu=1}^{K_3}  \omega_{\mu,K_2}  \hat \omega_{\nu,K_3} H_1^3(\xi_{\mu,K_2},\hat \xi_{\nu,K_3})  \\
    - \sum_{\mu=1}^{K_4} \sum_{\nu=1}^{K_5}  \hat \omega_{\mu,K_4} \omega_{\nu,K_5} H_2^3(\hat \xi_{\mu,K_4}, \xi_{\nu,K_5}) + \sum_{\nu_1=1}^K \sum_{\nu_2=1}^K \omega_{\nu_1,K} \omega_{\nu_2,K} H_3^3(\xi_{\nu_1,K},\xi_{\nu_2,K})
\end{multline*}
with the functions
\begin{equation*}
    H_1^3(\xi,\eta) := \psi^{p_1}_m(\eta) \, \eta \, \partial_t \psi^{p_1}_n(\xi \eta), \quad H_2^3(\xi,\eta) :=  \psi^{p_1}_m(\xi\eta) \, \xi \, \partial_t \psi^{p_1}_n(\xi),
\end{equation*}
\begin{equation*}
    H_3^3(\xi,\eta) := \psi^{p_1}_m(\eta) \eta \ln(1+\xi) \partial_t \psi^{p_1}_n(\xi \eta) + \psi^{p_1}_m(\xi\eta) \xi \ln(1+\eta) \partial_t \psi^{p_1}_n(\xi)
\end{equation*}
and the number of integration nodes $\frac{p_1+1}{2} \leq K_1 \in \N$, $\frac{p_1}{2} \leq K_2 \in \N$,  $p_1+\frac{1}{2} \leq K_3 \in \N$, $p_1 + \frac{1}{2} \leq K_4 \in \N$, $\frac{p_1+1}{2} \leq K_5 \in \N$ and $K \in \N$ sufficiently large. Note that all integrals of $I_{1,1}^3$ are calculated exactly except for the integral with the integrand $H_3^3$, where arguments similar to \cite[Corollary~3.1]{SteinbachZankJNUM2021} yield exponential convergence of the applied Gauß--Legendre quadratures with respect to the number $K$ of Gauß integration nodes.

\subsubsection{Case $k=\ell=N$} 

For this and the remaining cases, we state only the integrals, where the quadratures of Section~\ref{sec:Int} have to be applied. We calculate
\begin{align*} 
 M^{\mathcal H_T}_{N,N}[n, m] =& - \frac{h_N}{\pi} \int_0^1 \int_0^1 \psi^{p_N}_m(\eta) G_{N,N}^1(\xi,\eta) \partial_t \psi^{p_N}_n((1-\xi)\eta)  \, \mathrm d\xi \, \mathrm d\eta \\
  &\;- \frac{h_N}{\pi} \int_0^1 \int_0^1  \psi^{p_N}_m((1-\eta)\xi) G_{N,N}^2(\xi,\eta) \partial_t \psi^{p_N}_n(\xi)  \, \mathrm d\eta \, \mathrm d\xi  \\
  &\;- \frac{h_N }{\pi} \int_0^1  \int_0^1 \psi^{p_N}_m(\eta) \partial_t \psi^{p_N}_n(\xi) \ln\abs{\xi-\eta} \, \mathrm d\xi \, \mathrm d\eta \\
  &\;+\frac{h_N}{\pi}  \int_0^1  \int_0^1 \psi^{p_N}_m(1-\eta) \partial_t \psi^{p_N}_n(1-\eta\xi) \left[ \ln \eta + \ln(1+\xi) \right] \eta \, \mathrm d\xi \, \mathrm d\eta \\
  &\;+\frac{h_N}{\pi}  \int_0^1  \int_0^1 \psi^{p_N}_m(1-\xi\eta) \partial_t \psi^{p_N}_n(1-\xi) \left[ \ln \xi + \ln(1+\eta) \right] \xi \, \mathrm d\xi \, \mathrm d\eta
\end{align*}
with the functions
\begin{align*}
  F_{N,N}(s,t) &:= \ln \left[ \tan \frac{\pi (s+t)}{4T} (2T-s-t) \frac{ \tan \frac{\pi \abs{t-s}}{4T} }{ \abs{s-t}} \right], \\
  G_{N,N}^1(\xi,\eta) &:= F_{N,N}(t_{N-1}+(1-\xi)\eta h_N,t_{N-1}+\eta h_N) \eta,\\
  G_{N,N}^2(\xi,\eta) &:= F_{N,N}(t_{N-1}+\xi h_N,t_{N-1}+(1-\eta)\xi h_N) \xi,
\end{align*}
which do not have any singularity in the domain of integration. For the first and the second integral, we apply the tensor Gauß quadrature~\eqref{Int:GaussTensor}, for the third integral, the quadrature~\eqref{Int:GaussLogTensor} and for the last two integrals, the quadratures~\eqref{Int:GaussLegendre}, \eqref{Int:Gausslog}, as shown in Subsection~\ref{sec:NumHTMh:Fall11}.

\subsubsection{Cases, excluding $k=\ell=1$ and $k=\ell=N$}

We compute
\begin{multline*} 
 M^{\mathcal H_T}_{k,\ell}[n, m] = - \frac{h_\ell}{\pi} \int_0^1 \int_0^1 \psi^{p_\ell}_m(\eta) G_{k,\ell}^1(\xi,\eta) \partial_t \psi^{p_k}_n((1-\xi)\eta)  \, \mathrm d\xi \, \mathrm d\eta - \frac{h_\ell}{\pi} I^2_{k,\ell} \\
  - \frac{h_\ell}{\pi} \int_0^1 \int_0^1  \psi^{p_\ell}_m((1-\eta)\xi) G_{k,\ell}^2(\xi,\eta) \partial_t \psi^{p_k}_n(\xi)  \, \mathrm d\eta \, \mathrm d\xi  -  \delta_{k,1}\psi^{p_k}_n(0) \frac{2 h_\ell}{\pi} I^1_{\ell}
\end{multline*}
with the integrals
\begin{align*}
  I^1_{\ell} &:= \int_0^1 \psi^{p_\ell}_m(\eta) \ln \tan \frac{\pi (t_{\ell-1} + \eta h_\ell)}{4T} \mathrm d\eta, \\
  I^2_{k,\ell} &:=\int_0^1  \int_0^1 \psi^{p_\ell}_m(\eta) \partial_t \psi^{p_k}_n(\xi) \ln\abs{(t_{k-1}+\xi h_k)-(t_{\ell-1}+\eta h_\ell)} \, \mathrm d\xi \, \mathrm d\eta
\end{align*}
and with the functions
\begin{align*}
  F_{k,\ell}(s,t) &:= \ln \left[ \tan \frac{\pi (s+t)}{4T} \frac{ \tan \frac{\pi \abs{t-s}}{4T} }{ \abs{s-t}} \right], \\
  G_{k,\ell}^1(\xi,\eta) &:= F_{k,\ell}(t_{k-1}+(1-\xi)\eta h_k,t_{\ell-1}+\eta h_\ell) \eta ,\\
  G_{k,\ell}^2(\xi,\eta) &:= F_{k,\ell}(t_{k-1}+\xi h_k,t_{\ell-1}+(1-\eta)\xi h_\ell) \xi,
\end{align*}
which do not have any singularity in the domain of integration, i.e., the tensor Gauß quadrature~\eqref{Int:GaussTensor} is applied to these integrands. The singular parts $I^1_{\ell}, I^2_{k,\ell}$ are treated differently corresponding to the indices $k,\ell.$

First, we consider the part $I^1_{\ell}$. The case $\ell=1$ is excluded, since $I^1_{\ell}$ contributes only for $k=1$ due to $\delta_{k,1} = 0$ for $k>1$. For $\ell>1$, the term $I^1_{\ell}$ is not singular. Thus, we apply the Gauß--Legendre quadrature~\eqref{Int:GaussLegendre} to
\begin{equation*}
  I^1_{\ell} = \int_0^1 \psi^{p_\ell}_m(\eta) \ln \tan \frac{\pi (t_{\ell-1} + \eta h_\ell)}{4T} \mathrm d\eta
\end{equation*}
for $\ell > 1.$

Second, we consider the part $I^2_{k,\ell}$. We distinguish four cases:
\begin{enumerate}
  \item case $k=\ell$: We compute
        \begin{equation*}
            I^2_{k,\ell} = \ln h_\ell \left( \psi^{p_\ell}_n(1) - \psi^{p_\ell}_n(0) \right) \int_0^1 \psi^{p_\ell}_m(\eta) \mathrm d\eta + \int_0^1 \int_0^1 \psi^{p_\ell}_m(\eta) \partial_t \psi^{p_\ell}_n(\xi) \ln\abs{\xi-\eta} \mathrm d\xi \mathrm d\eta,
        \end{equation*}
        where the quadratures~\eqref{Int:GaussLegendre}, \eqref{Int:GaussLogTensor} are applied.
  \item case $k=\ell+1$: The equality
        \begin{multline*}
            I^2_{k,\ell} = \int_0^1 \int_0^1 \psi^{p_\ell}_m(1-\eta) \partial_t \psi^{p_k}_n(\eta \xi) \eta \left( \ln \eta + \ln (\xi h_k + h_\ell) \right)  \mathrm d\xi \mathrm d\eta \\
            + \int_0^1 \int_0^1 \psi^{p_\ell}_m(1-\xi +\eta\xi) \partial_t \psi^{p_k}_n(\xi) \xi \left( \ln \xi + \ln ((1-\eta) h_\ell + h_k) \right) \mathrm d\xi \mathrm d\eta
        \end{multline*}  
        holds true, where we use the quadratures~\eqref{Int:GaussLegendre}, \eqref{Int:Gausslog}, and \eqref{Int:GaussTensor}.
  \item case $k+1=\ell$: We have that
        \begin{multline*}
            I^2_{k,\ell} = \int_0^1 \int_0^1 \psi^{p_\ell}_m((1-\xi)\eta) \partial_t \psi^{p_k}_n(1-\eta) \eta \left( \ln \eta + \ln (h_k + (1-\xi) h_\ell) \right)  \mathrm d\xi \mathrm d\eta \\
            + \int_0^1 \int_0^1 \psi^{p_\ell}_m(\eta) \partial_t \psi^{p_k}_n(1-\eta+\xi\eta) \eta \left( \ln \eta + \ln ((1-\xi) h_k + h_\ell) \right) \mathrm d\xi \mathrm d\eta
        \end{multline*}  
        holds true, where we use again the quadratures~\eqref{Int:GaussLegendre}, \eqref{Int:Gausslog}, and \eqref{Int:GaussTensor}.
  \item otherwise: We apply the tensor Gauß quadrature~\eqref{Int:GaussTensor} to
        \begin{equation*}
            I^2_{k,\ell} = \int_0^1 \int_0^1 \psi^{p_\ell}_m(\eta) \partial_t \psi^{p_k}_n(\xi)  \ln \abs{ (t_{k-1} + \xi h_k) - (t_{\ell-1} + \eta h_\ell) }  \mathrm d\xi \mathrm d\eta,
        \end{equation*}
        as the integral is not singular.
\end{enumerate}

\subsection{Local Matrix $A^{\mathcal H_T}_{k,\ell}$}  \label{sec:NumHTAh}

In this subsection, the computation of the local matrix $A^{\mathcal H_T}_{k,\ell}$ in \eqref{NumHT:Aht_local} is investigated for fixed elements $\tau_k$, $\tau_\ell$ and related polynomial degrees $p_k$, $p_\ell$ with $k, \ell \in \{1,\dots,N\}$. To calculate the local matrix $A^{\mathcal H_T}_{k,\ell}$, replace in Subsection~\ref{sec:NumHTMh} the function $\psi^{p_\ell}_m(\cdot)$ with the function $\frac{1}{h_\ell}\partial_t \psi^{p_\ell}_m(\cdot)$ in all occurring quantities.

\subsection{Local Matrix $B^{\mathcal H_T}_{k,\ell}$}  \label{sec:NumHTBh}

In this subsection, the computation of the local matrix $B^{\mathcal H_T}_{k,\ell}$ in \eqref{NumHT:Bht_local} is investigated for fixed elements $\tau_k$, $\tau_\ell$ and related polynomial degrees $p_k$, $p_\ell$ with $k, \ell \in \{1,\dots,N\}$. To calculate the second line in \eqref{NumHT:Bht_local}, replace in Subsection~\ref{sec:NumHTMh} the function $\psi^{p_\ell}_m(\cdot)$ with the function $\frac{1}{h_\ell}\partial_t \psi^{p_\ell}_m(\cdot)$ and the function $\partial_t \psi^{p_k}_n(\cdot)$ with the function $\frac{1}{h_k}\partial_{tt} \psi^{p_k}_n(\cdot)$ in all occurring quantities, which are related to the last line in \eqref{NumHT:Mht_local}. Thus, it remains to compute the integrals $J_{k,\ell}^0[m]$ and $J_{k,\ell}^1[m]$ given in \eqref{NumHT:Bht_local_J0} and \eqref{NumHT:Bht_local_J1}, respectively, where $m=1,\dots,p_\ell+1$.

As in Subsection~\ref{sec:NumHTMh}, we distinguish three different cases for the element indices $k$, $\ell$, which correspond to the singularities of the integrands in  \eqref{NumHT:Bht_local_J0} and \eqref{NumHT:Bht_local_J1}. For this purpose, let $m \in \{1,\dots,p_\ell+1\}$ be fixed. We state only the final results, since their derivation is given in a simpler or similar way as in Subsection~\ref{sec:NumHTMh}. These final results are formulated as integrals such that these integrals can be directly calculated or approximated by the Gauß--Legendre quadrature~\eqref{Int:GaussLegendre} and the Gauß--Jacobi quadrature~\eqref{Int:Gausslog}.

\subsubsection{Case $k=\ell=1$} 

We calculate
\begin{equation*} 
  J_{k,\ell}^0[m] = \int_0^1 \partial_t \psi^{p_1}_m(\eta) \mathcal F_{1,1}(0, \eta h_1) \mathrm d\eta+ 2 \int_0^1 \partial_t \psi^{p_1}_m(\eta) \ln \eta \mathrm d\eta  + 2 \ln h_1 \big( \psi^{p_1}_m(1) - \psi^{p_1}_m(0) \big) 
\end{equation*}
and
\begin{multline*} 
  J_{k,\ell}^1[m] = \int_0^1 \partial_t \psi^{p_1}_m(\eta) \mathcal F_{1,1}(t_1, \eta h_1) \mathrm d\eta + \int_0^1 \partial_t \psi^{p_1}_m(\eta) \ln(1+\eta) \mathrm d\eta \\
  + \int_0^1 \partial_t \psi^{p_1}_m(1-\eta) \ln\eta \mathrm d\eta + 2 \ln h_1 \big( \psi^{p_1}_m(1) - \psi^{p_1}_m(0) \big) 
\end{multline*}
with the function
\begin{equation*}
  \mathcal F_{1,1}(s,t) := \ln \left[ \frac{ \tan \frac{\pi (s+t)}{4T}}{s+t} \frac{ \tan \frac{\pi \abs{t-s}}{4T} }{ \abs{s-t}} \right],
\end{equation*}
which does not have any singularity in the domain of integration. For these integrals, we apply the quadratures~\eqref{Int:GaussLegendre}, \eqref{Int:Gausslog}, as shown in Subsection~\ref{sec:NumHTMh:Fall11}.

\subsubsection{Case $k=\ell=N$}

We compute
\begin{multline*}
  J_{k,\ell}^0[m] = \int_0^1 \partial_t \psi^{p_N}_m(\eta) \mathcal F_{N,N}(t_{N-1}, t_{N-1}+\eta h_N) \mathrm d\eta \\
  + \int_0^1 \partial_t \psi^{p_N}_m(\eta) \ln \eta \mathrm d\eta - \int_0^1 \partial_t \psi^{p_N}_m(\eta) \ln (2-\eta) \mathrm d\eta 
\end{multline*}
and
\begin{equation*} 
  J_{k,\ell}^1[m] = 0
\end{equation*}
with the function
\begin{equation*}
  \mathcal F_{N,N}(s,t) := \ln \left[ \tan \frac{\pi (s+t)}{4T} (2T-s-t) \frac{ \tan \frac{\pi \abs{t-s}}{4T} }{ \abs{s-t}} \right],
\end{equation*}
which does not have any singularity in the domain of integration. For these integrals, we apply again the quadratures~\eqref{Int:GaussLegendre}, \eqref{Int:Gausslog}, as shown in Subsection~\ref{sec:NumHTMh:Fall11}.

\subsubsection{Cases, excluding $k=\ell=1$ and $k=\ell=N$}

We have
\begin{align*} 
 J_{k,\ell}^0[m] &= \int_0^1 \partial_t \psi^{p_\ell}_m(\eta) \mathcal F_{k,\ell}(t_{k-1}, t_{\ell-1}+\eta h_\ell) \mathrm d\eta + J_{k,\ell}^{0,\mathrm{sing}}[m], \\
 J_{k,\ell}^1[m] &= \int_0^1 \partial_t \psi^{p_\ell}_m(\eta) \mathcal F_{k,\ell}(t_k, t_{\ell-1}+\eta h_\ell) \mathrm d\eta + J_{k,\ell}^{1,\mathrm{sing}}[m]
\end{align*}
with the function
\begin{equation*}
  \mathcal F_{k,\ell}(s,t) := \ln \left[ \tan \frac{\pi (s+t)}{4T} \frac{ \tan \frac{\pi \abs{t-s}}{4T} }{ \abs{s-t}} \right],
\end{equation*}
which does not have any singularity in the domain of integration, and the terms
\begin{equation*}
   J_{k,\ell}^{0,\mathrm{sing}}[m] =
   \begin{cases}
     \ln h_\ell \big( \psi^{p_\ell}_m(1) - \psi^{p_\ell}_m(0) \big) + \int_0^1 \partial_t \psi^{p_\ell}_m(\eta) \ln \eta \mathrm d\eta, & k=\ell, \\
     \ln h_\ell \big( \psi^{p_\ell}_m(1) - \psi^{p_\ell}_m(0) \big) + \int_0^1 \partial_t \psi^{p_\ell}_m(1-\eta) \ln \eta \mathrm d\eta, & k=\ell+1, \\
     \int_0^1 \partial_t \psi^{p_\ell}_m(\eta) \ln(h_k+\eta h_\ell) \mathrm d\eta, & k+1=\ell, \\
     \int_0^1 \partial_t \psi^{p_\ell}_m(\eta) \ln \abs{t_{k-1} - (t_{\ell-1}+\eta h_\ell)} \mathrm d\eta, & \text{otherwise,}     
   \end{cases}
\end{equation*}
and
\begin{equation*}
   J_{k,\ell}^{1,\mathrm{sing}}[m] =
   \begin{cases}
     \ln h_\ell \big( \psi^{p_\ell}_m(1) - \psi^{p_\ell}_m(0) \big) + \int_0^1 \partial_t \psi^{p_\ell}_m(1-\eta) \ln \eta \mathrm d\eta, & k=\ell, \\
     \int_0^1 \partial_t \psi^{p_\ell}_m(\eta) \ln(h_k+(1-\eta) h_\ell) \mathrm d\eta, & k=\ell+1, \\
     \ln h_\ell \big( \psi^{p_\ell}_m(1) - \psi^{p_\ell}_m(0) \big) + \int_0^1 \partial_t \psi^{p_\ell}_m(\eta) \ln \eta \mathrm d\eta, & k+1=\ell, \\
     \int_0^1 \partial_t \psi^{p_\ell}_m(\eta) \ln \abs{t_k - (t_{\ell-1}+\eta h_\ell)} \mathrm d\eta, & \text{otherwise.}     
   \end{cases}
\end{equation*}
Again, we apply the quadratures~\eqref{Int:GaussLegendre}, \eqref{Int:Gausslog} to approximate $J_{k,\ell}^0[m]$ and $J_{k,\ell}^1[m].$

\subsection{Exponential Convergence of the Proposed Quadrature Schemes}

In this subsection, we summarize the quality of the quadrature schemes proposed in Subsection~\ref{sec:NumHTMh}, \ref{sec:NumHTAh}, \ref{sec:NumHTBh}.

\begin{theorem} \label{NumHT:Thm:Konvergenz}
  Let the mesh \eqref{Einf:Netz} fulfill the assumption $\max_\ell h_\ell \leq T/2$. Further, choose all integration orders of the Gauß--Jacobi quadratures~\eqref{Int:Gausslog}, \eqref{Int:GaussLogTensor} such that the related integrals in Subsection~\ref{sec:NumHTMh}, \ref{sec:NumHTAh}, \ref{sec:NumHTBh} are calculated in an exact way. Then, all Gauß--Legendre quadratures~\eqref{Int:GaussLegendre}, \eqref{Int:GaussTensor} applied in Subsection~\ref{sec:NumHTMh}, \ref{sec:NumHTAh}, \ref{sec:NumHTBh} converges exponentially with respect to the Gauß integration nodes. In other words, the entries of the matrices $M^{\mathcal H_T}$, $A^{\mathcal H_T}$, $B^{\mathcal H_T}$ are computable to high float point accuracy.
\end{theorem}
\begin{proof}
  We apply the same arguments as in \cite[Corollary~3.1]{SteinbachZankJNUM2021}.
\end{proof}

\section{Numerical Examples} \label{sec:Num}

In this section, we give numerical examples for the assembling of the matrices $M^{\mathcal H_T}$, $A^{\mathcal H_T}$, $B^{\mathcal H_T}$, given in \eqref{Einf:Mh}, \eqref{Einf:Ah}, \eqref{Einf:Bh}, respectively. Further, we give a numerical example for the ordinary differential equation~\eqref{Einf:Waerme} related to the heat equation, using an $h$- and $hp$-FEM as discretization schemes.

For this purpose, we fix the choice of the shape functions. In this section, we use the Lobatto polynomials (or integrated Legendre polynomials) as hierarchical shape functions on the reference element $[0,1]$. In greater detail, for a given local polynomial degree $p \in \N$, we set
\begin{equation*}
 \psi_1^p(\xi) = 1 - \xi, \qquad \psi_2^p(\xi) = \xi, \qquad \psi_m^p(\xi) = \int_0^\xi L_{m-2}(\zeta) \mathrm d\zeta \quad \text{ for } m \geq 3, \quad \xi \in [0,1],
\end{equation*}
where $L_{m}$ denotes the $m$-th Legendre polynomial on $[0,1]$, see \cite[Subsection~3.1.4]{Schwabphp1998}.

\subsection{Numerical Integration}

In this subsection, we show a numerical example for the quadrature schemes presented in Section~\ref{sec:NumHT} to calculate the entries of the matrices $M^{\mathcal H_T}$, $A^{\mathcal H_T}$, $B^{\mathcal H_T}$. For this purpose, for $T=10$, we fix the non-uniform mesh by $t_0=0$, $t_\ell = 2^{\ell-N}T$ for $\ell = 1,\dots, N$ with $N=6$ and the polynomial degree vector by $\boldsymbol{p}=(2,2,2,2,2,2)$, i.e., the degree vector is uniform. Thus, the number of the degrees of freedom is $M=13$. In this situation, we are able to calculate the matrices $M^{\mathcal H_T}$, $A^{\mathcal H_T}$, $B^{\mathcal H_T}$ as proposed in \cite[Subsection~2.2]{ZankCMAM2021}, which serve as reference values for these matrix entries. We compare these reference values with the values, calculated as proposed in Section~\ref{sec:NumHT}, when increasing the number of integration nodes $K=K_1=K_2$ of the Gauß--Legendre quadratures~\eqref{Int:GaussLegendre}, \eqref{Int:GaussTensor}. Here, we choose the orders of the Gauß--Jacobi quadratures~\eqref{Int:Gausslog}, \eqref{Int:GaussLogTensor} such that all related integrals are computed in an exact way, i.e., no additional approximation occurs. The errors are measured in $\norm{\cdot}_{\max}$ defined by $\norm{A}_{\max} := \max_{i,j} \abs{A[i,j]}$ for a matrix $A$ with entries $A[i,j]$.

In Figure~\ref{Num:Fig:NumInt}, we observe exponential convergence with respect to the number of Gauß integration nodes $K$ for the matrix entries of $M^{\mathcal H_T}$, $A^{\mathcal H_T}$, $B^{\mathcal H_T}$ calculated as proposed in Section~\ref{sec:NumHT}. Note that this exponential convergence is in accordance with Theorem~\ref{NumHT:Thm:Konvergenz}.

\begin{figure}
\begin{center}
    \includegraphics[scale=1.0]{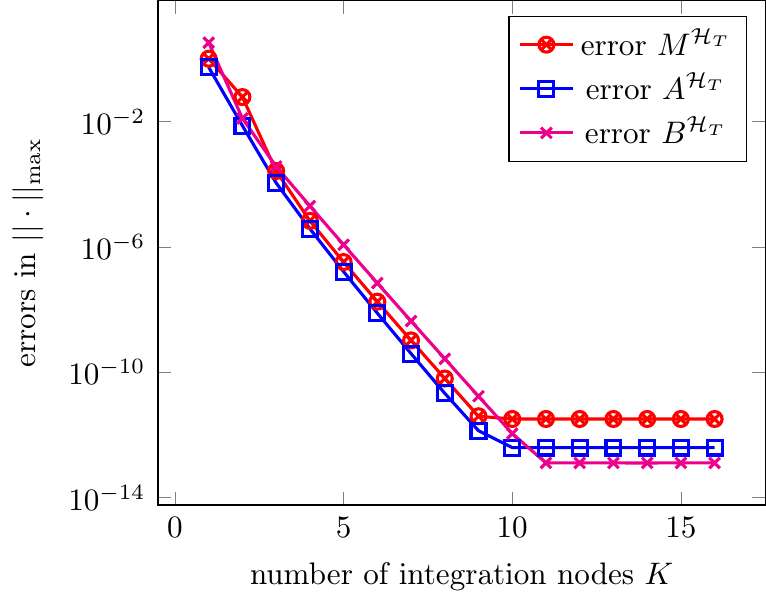}
    \caption{Numerical results of the quadrature schemes of Section~\ref{sec:NumHT} with $K$ Gauß--Legendre points per coordinate direction with a non-uniform mesh for the matrices $M^{\mathcal H_T}$, $A^{\mathcal H_T}$, $B^{\mathcal H_T}$ for piecewise quadratic functions.}
    \label{Num:Fig:NumInt}
\end{center}
\end{figure}

\subsection{ODE related to the Heat Equation}

In this subsection, we give a numerical example for the ordinary differential equation~\eqref{Einf:Waerme}, which is related to the heat equation. We use the discrete variational formulation~\eqref{Einf:WaermeVFDisk} for different choices of the meshes $\mathcal T^N$ and polynomial degrees $\boldsymbol{p}$. 

First, we apply an $h$-FEM. For given $N \in \N$, $N\geq 2$, we consider the uniform mesh defined by $t_\ell=T\frac{\ell}{N}$, $\ell=0,\dots,N$, i.e., a uniform refinement strategy is used when $N$ doubles. The vector of the polynomial degrees is uniform, i.e., $\boldsymbol{p}=(p,\dots,p)$ with $p=2$ or $p=10$.

Second, we apply an $hp$-FEM. For given $N \in \N$, $N\geq 2$, the geometric mesh is defined by $t_0=0$, $t_\ell=T \sigma^{N-\ell}$ for $\ell \in \{1,\dots,N\}$ with the
grading parameter $\sigma = 0.17$, see \cite[p. 96]{Schwabphp1998}. The vector of the polynomial degrees $\boldsymbol{p} = (p_1,\dots,p_N)$ is chosen as $p_\ell = \ell$ for $\ell \in \{1,\dots,N\}$.

Further, we set $K=K_1=K_2=20$ for the number of integration nodes of the Gauß--Legendre quadratures~\eqref{Int:GaussLegendre}, \eqref{Int:GaussTensor}. In addition, we choose the orders of the Gauß--Jacobi quadratures~\eqref{Int:Gausslog}, \eqref{Int:GaussLogTensor} such that all related integrals are computed in an exact way.

Last, we measure the error in the norm $\| \cdot \|_{H^{1/2}_{0,}(0,T)}$, which is hardly computable. Thus, we use the approximation
\begin{equation*}
  [v]_{H^{1/2}_{0,}(0,T)} := \sqrt{ \| v \|_{L^2(0,T)} \| \partial_t v \|_{L^2(0,T)} }, \quad  v \in H^1_{0,}(0,T),
\end{equation*}
which is an upper bound for $C\| \cdot \|_{H^{1/2}_{0,}(0,T)}$ with $C>0$, due to the interpolation inequality, see \cite[p. 23]{LM11968}.

In Figure~\ref{Num:Fig:GDGParabolisch}, we state the numerical results of the discrete variational formulation~\eqref{Einf:WaermeVFDisk} for the exact solution $u(t) = t^{3/4},$ $t \in [0,T]$ for $T=1$. For the $h$-FEM with $p=2$ or $p=10$, we observe the reduced convergence due to the low regularity $u \in H^{5/4-\varepsilon}(0,T)$ with $\varepsilon \in (0,1)$. For the $hp$-FEM, exponential convergence with respect to the number of degrees of freedom $M$ is achieved. Note that this exponential convergence is proven in \cite{PerugiaSchwabZankhp2022}. This numerical examples show that the proposed quadrature schemes in Section~\ref{sec:NumHT} provide the matrices $M^{\mathcal H_T}$, $A^{\mathcal H_T}$, $B^{\mathcal H_T}$ in high float point accuracy.

\begin{figure}
\begin{center}
    \includegraphics[scale=1.0]{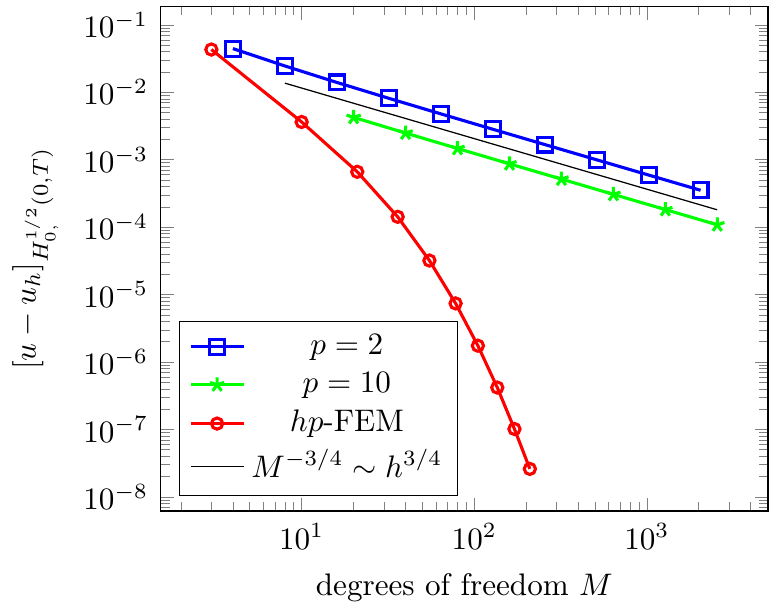}
    \caption{Numerical results of the finite element discretization \eqref{Einf:WaermeVFDisk} for an $h$-FEM with $p=1$ or $p=10$, and for an $hp$-FEM for the exact solution $u(t) = t^{3/4}$, $t \in [0,T]$, with $T=1$.}
    \label{Num:Fig:GDGParabolisch}
\end{center}
\end{figure}

\section{Conclusion} \label{sec:Zum}

The matrices $M^{\mathcal H_T}$, $A^{\mathcal H_T}$, $B^{\mathcal H_T}$ occur for finite element discretizations of the heat or wave equation. The aim of this work was to realize these matrices for arbitrary polynomial degrees, which are used for, e.g., a temporal $hp$-FEM. First, we stated weakly singular integral representations of the modified Hilbert transformation $\mathcal H_T$. These integral representations formed the basis of the quadrature schemes to calculate the matrices $M^{\mathcal H_T}$, $A^{\mathcal H_T}$, $B^{\mathcal H_T}$. All occurring singular integrals were reformulated such that they can be calculated to machine precision. In the last part, exponential convergence with respect to the number of Gauß integration nodes was observed in numerical examples. Moreover, an $hp$-FEM for an ordinary differential equation, which is related to the heat equation, showed the potential of the proposed quadrature schemes.

\bibliographystyle{acm}
\bibliography{zankliteratur}

\begin{thebibliography}{10}

\bibitem{BangerthGeigerRannacher2010}
{\sc Bangerth, W., Geiger, M., and Rannacher, R.}
\newblock Adaptive {G}alerkin finite element methods for the wave equation.
\newblock {\em Comput. Methods Appl. Math. 10}, 1 (2010), 3--48.

\bibitem{ErnGuermondFEMI2021}
{\sc Ern, A., and Guermond, J.-L.}
\newblock {\em Finite elements {I}---{A}pproximation and interpolation},
  vol.~72 of {\em Texts in Applied Mathematics}.
\newblock Springer, Cham, 2021.

\bibitem{Gautschi2012Log}
{\sc Gautschi, W.}
\newblock Numerical integration over the square in the presence of
  algebraic/logarithmic singularities with an application to aerodynamics.
\newblock {\em Numer. Algorithms 61}, 2 (2012), 275--290.

\bibitem{Gautschi2012LogErratum}
{\sc Gautschi, W.}
\newblock Erratum to: ``{Numerical} integration over the square in the presence
  of algebraic/logarithmic singularities with an application to aerodynamics''.
\newblock {\em Numer. Algorithms 64}, 4 (2013), 759.

\bibitem{Ladyzhenskaya1985}
{\sc Ladyzhenskaya, O.}
\newblock {\em The boundary value problems of mathematical physics}, vol.~49 of
  {\em Applied Mathematical Sciences}.
\newblock Springer-Verlag, New York, 1985.

\bibitem{SteinbachLangerBuch2019}
{\sc Langer, U., and Steinbach, O.}, Eds.
\newblock {\em Space-time methods---applications to partial differential
  equations}, vol.~25 of {\em Radon Series on Computational and Applied
  Mathematics}.
\newblock De Gruyter, Berlin, 2019.

\bibitem{LangerZankSISC2021}
{\sc Langer, U., and Zank, M.}
\newblock Efficient direct space-time finite element solvers for parabolic
  initial-boundary value problems in anisotropic {S}obolev spaces.
\newblock {\em SIAM J. Sci. Comput. 43}, 4 (2021), A2714--A2736.

\bibitem{LM11968}
{\sc Lions, J.-L., and Magenes, E.}
\newblock {\em Probl\`emes aux limites non homog\`enes et applications. {V}ol.
  1}.
\newblock Travaux et Recherches Math\'ematiques, No. 17. Dunod, Paris, 1968.

\bibitem{LoescherSteinbachZankDD}
{\sc L\"oscher, R., Steinbach, O., and Zank, M.}
\newblock Numerical results for an unconditionally stable space-time finite
  element method for the wave equation.
\newblock In {\em Domain Decomposition Methods in Science and Engineering
  {XXVI}}, vol.~145 of {\em Lect. Notes Comput. Sci. Eng.} Springer, Cham,
  2022, pp.~587--594.

\bibitem{LoescherSteinbachZankWelleTheorie}
{\sc L\"oscher, R., Steinbach, O., and Zank, M.}
\newblock An unconditionally stable space-time finite element method for the
  wave equation.
\newblock {\em In preparation\/} (2022).

\bibitem{Necas1967}
{\sc Ne{\v{c}}as, J.}
\newblock {\em Les m{\'e}thodes directes en th{\'e}orie des {\'e}quations
  elliptiques}.
\newblock Paris: Masson et Cie; Prague: Academia, 1967.

\bibitem{PerugiaSchwabZankhp2022}
{\sc Perugia, I., Schwab, C., and Zank, M.}
\newblock Exponential convergence of {$hp$}-time-stepping in space-time
  discretizations of parabolic {PDEs}.
\newblock [math.NA] arXiv:2203.11879, arXiv.org, 2022.

\bibitem{Schwabphp1998}
{\sc Schwab, C.}
\newblock {\em {$p$}- and {$hp$}-finite element methods}.
\newblock Numerical Mathematics and Scientific Computation. The Clarendon
  Press, Oxford University Press, New York, 1998.

\bibitem{SteinbachMissoni2022}
{\sc Steinbach, O., and Missoni, A.}
\newblock A note on a modified {H}ilbert transform.
\newblock {\em Applicable Analysis\/} (2022), 1--8.

\bibitem{SteinbachZankETNA2020}
{\sc Steinbach, O., and Zank, M.}
\newblock Coercive space-time finite element methods for initial boundary value
  problems.
\newblock {\em Electron. Trans. Numer. Anal. 52\/} (2020), 154--194.

\bibitem{SteinbachZankJNUM2021}
{\sc Steinbach, O., and Zank, M.}
\newblock A note on the efficient evaluation of a modified {Hilbert}
  transformation.
\newblock {\em J. Numer. Math. 29}, 1 (2021), 47--61.

\bibitem{ThomeeBuch2006}
{\sc Thom\'{e}e, V.}
\newblock {\em Galerkin finite element methods for parabolic problems},
  second~ed., vol.~25 of {\em Springer Series in Computational Mathematics}.
\newblock Springer-Verlag, Berlin, 2006.

\bibitem{ZankDissBuch2020}
{\sc Zank, M.}
\newblock {\em Inf-Sup Stable Space-Time Methods for Time-Dependent Partial
  Differential Equations}, vol.~36 of {\em Monographic Series TU Graz:
  Computation in Engineering and Science}.
\newblock 2020.

\bibitem{ZankCMAM2021}
{\sc Zank, M.}
\newblock An exact realization of a modified {H}ilbert transformation for
  space-time methods for parabolic evolution equations.
\newblock {\em Comput. Methods Appl. Math. 21}, 2 (2021), 479--496.

\bibitem{ZankWelleLoeser}
{\sc Zank, M.}
\newblock Efficient direct space-time finite element solvers for the wave
  equation in second-order formulation.
\newblock {\em In preparation\/} (2022).

\bibitem{ZankWaermeLoeserICOSAHOM2022}
{\sc Zank, M.}
\newblock High-order discretisations and efficient direct space-time finite
  element solvers for parabolic initial-boundary value problems.
\newblock {\em Accepted for the Proceedings of Spectral and High Order Methods
  for Partial Differential Equations ICOSAHOM 2020+1\/} (2022).

\end{thebibliography}

\end{document}